\documentclass[psamsfonts]{amsart}

\usepackage{amssymb,amsfonts}
\usepackage[all,arc]{xy}
\usepackage{enumerate}
\usepackage{mathrsfs}
\newtheorem{thm}{Theorem}[section]
\newtheorem{cor}[thm]{Corollary}

\theoremstyle{definition}
\newtheorem{defn}[thm]{Definition}

\newtheorem{exmp}[thm]{Example}

\theoremstyle{remark}

\makeatletter
\let\c@equation\c@thm
\makeatother
\numberwithin{equation}{section}

\title{On sequences in $2$-normed spaces}

\author{\; \; \; Sibel Ersan and Huseyin Cakalli\\ Maltepe University, \.{I}stanbul\\ Turkey}
\address{Sibel Ersan\\
            Faculty of Engineering and Natural Sciences, Maltepe University, Marmara E\u{g}\.{I}t\.{I}m K\"oy\"u, TR 34857, Maltepe, \.{I}stanbul-Turkey\\ Phone: (+90216) 6261050 ext:2396, fax: (+90216) 6261131}
\email{sibelersan@maltepe.edu.tr; sibelersan@gmail.com}
\address{H\"usey\.{I}n \c{C}akall\i \\
          Maltepe University, Department of Mathematics, Marmara E\u{g}\.{I}t\.{I}m K\"oy\"u, TR 34857, Maltepe, \.{I}stanbul-Turkey, Phone:(+90216)6261050 ext:2248, fax: (+90216) 6261113}
\email{hcakalli@maltepe.edu.tr; hcakalli@gmail.com}

\date{\today}

\begin{document}

\begin{abstract}
A function $f$ defined on a $2$-normed space $ (X,||.,.||)$  is ward continuous if it preserves quasi-Cauchy sequences where a sequence $(x_n)$  of points in $X$ is called quasi-Cauchy if  $lim_{n\rightarrow\infty}||\Delta x_{n},z||=0$ for every $z\in X$. Some other kinds of continuties are also introduced via quasi-Cauchy sequences in $2$-normed spaces. It turns out that uniform limit of ward continuous functions is again ward continuous.

\end{abstract}

\maketitle

\section{Introduction}
The concept of continuity and any concept involving continuity play a very important role not only in pure mathematics but also in other branches of sciences involving mathematics especially in computer science, information theory, biological science.

About 90 years ago, Menger (\cite{menger}) introduced a notion called a generalized metric. But many mathematicians had not paid attentions to Menger
theory about generalized metrics. Very few mathematicians, for example, A.Wald, L. M. Blumenthal, W.A. Wilson, O. Haupt, C. Pauc, etc have
developed Menger's idea. First these researches had a very close connection with the direct method of variational calculus (for details, see \cite{blu,blu1} and \cite{pauc}). On the other hand, in 1938 Vulich (\cite{vulic}) introduced a notion of higher dimensional norm in linear spaces. Unfortunately, this study had been neglected by many analysists for along time. A Froda's work (\cite{froda}) appeared in 1958. Then a new development began with 1962 by G\"ahler (\cite{gohler}, \cite{g2}, and \cite{g3}).

The notions of a two norm, and a two metric have been extensively studied in (\cite{iseki, gunawan, ChuandParkandParkTheAleksandrovprobleminlinear2normedspaces, ist}) in which a lot for the extension of this branch of mathematics is contributed. Recently many mathematicians came out with results in $2$-normed spaces, analogous with that in classical normed spaces and Banach spaces (see for example \cite{MohiuddineSomenewresultsonapproximationinfuzzy2-normedspaces}, \cite{raji}, and \cite{DasandSavasandBhuniaSantanuTwovaluedmeasureandsomenewdoublesequencespacesin2-normedspaces}).

The concepts of ward continuity of a real function and ward compactness of a subset $E$ of $\textbf{R}$ are introduced by Cakalli in \cite{cakalli} (see also \cite{DikandCanak}, and \cite{burton}). A real function $f$ is called ward continuous on $E$ if the sequence $(f(x_n))$ is quasi-Cauchy whenever $x=(x_n)$ is a quasi-Cauchy sequence of points in $E$.

The aim of this paper is to investigate quasi-Cauchy sequences in $2$-normed spaces, and prove interesting theorems.

\section{Preliminaries}
Now we give some notation and definitions which will be needed in the paper. Throughout this paper, $\textbf{N}$ and $\textbf{R}$ will denote the set of positive integers and the set of real numbers, respectively. First we recall the definition of a $2$-normed space.

\begin{defn} \label{Definitionofatwonormedspace} (\cite{gohler})
Let $X$ be a real linear space with $\dim X >1$ and \; \; \; \; \; $||.,.||:X^{2}\rightarrow \textbf{R}$ a function. Then $(X,||.,.||)$ is called a linear $2$-normed space if

\begin{enumerate}
	\item $||x,y||=0\Leftrightarrow$ x and y are linearly dependent,
	\item $||x,y||=||y,x||$,
	\item $||\alpha x,y||=\left|\alpha\right|||x,y||$,
	\item $||x,y+z||\leq||x,y||+||x,z||$
\end{enumerate}
for $\alpha \in \textbf{R}$ and $x,y,z\in X$. The function $||.,.||$ is called the $2$-norm on $X$.

Observe that in any $2$-normed space $ (X,||.,.||)$ we have $ ||.,.||$ is nonnegative, $||x-z,x-y||=||x-z,y-z||$, and $\forall x,y\in X, \alpha\in \mathbf{R}$ $||x,y+\alpha x||=||x,y||$. Throughout this paper by $X$ we will mean a $2$-normed space with a two norm $||.,.||$.
\end{defn}
A classical example is the $2$-normed space $X=\mathbf{R}^{2}$ with the two norm $||.,.||$ defined  by $||a,b||=\left|a_{1}b_{2}-a_{2}b_{1}\right|$ where $a=(a_{1},a_{2})$, $b=(b_{1},b_{2})\in \mathbf{R}^{2}$. This is
the area of the parallelogram determined by the vectors $a$ and $b$.

A sequence $(x_{n})$ of points in $X$ is said to be convergent to an element $x\in{X}$ if $lim_{n\rightarrow\infty}||x_{n}-x,z||=0$ for every $z\in X$. This is denoted by $lim_{n\rightarrow\infty} ||x_{n}, z|| = ||x, z||$. A sequence $(x_n)$  of points in $X$ is called Cauchy if  $lim_{n,m\rightarrow\infty}||x_{n}-x_{m}, z||=0$ for every $z\in X$ (see \cite{gohler}).

A sequence of functions $(f_n)$ is said to be uniformly convergent to a function $f$ on a subset $E$ of $X$ if for each $\epsilon>0$, an integer $N$ can be found such that $||f_n(x)-f(x),z||<\epsilon$ for $n\geq N$ and for all $x,z\in X$.
	
\section{Quasi-Cauchy sequences in $2$-normed spaces}

In this section we investigate the notions of ward continuity and ward compactness on a $2$-normed space. Now we give some definitions that are used to describe the subject.
\begin{defn}  A sequence $(x_n)$  of points in a $2$-normed space $ (X,||.,.||)$ is called quasi-Cauchy if  $lim_{n\rightarrow\infty}||\Delta x_{n},z||=0$ for every $z\in X$ where $\Delta x_{n}=x_{n+1}-x_{n}$ for every $n\in{\textbf{N}}$.
\end{defn}

We note that any convergent sequence is quasi-Cauchy, and any Cauchy sequence is quasi-Cauchy. Any subsequence of a Cauchy sequence is Cauchy.
The analogous property fails for quasi-Cauchy sequences. A counter example is provided in the following.

\begin{exmp} The subsequence $(a_{n^{2}})=(n,n)$ of the quasi-Cauchy sequence $(a_n)=(\sqrt{n},\sqrt{n})$ of points in the $2$-normed space $R^2$ is not quasi-Cauchy.
\end{exmp}

\begin{defn}  A subset $E$ of $X$ is called ward compact if any sequence of points in $E$ has a quasi-Cauchy subsequence.
\end{defn}

First, we note that any finite subset of $X$ is ward compact, union of two ward compact subsets of $X$ is ward compact and intersection of ward compact subsets of $X$ is ward compact. Furthermore any subset of a ward compact set is ward compact.

For any given two norm we can define norms by using two norm values. Consider the norm $||.||$ defined on a linear $2$-normed space $(X,||.,.||)$ by the function $$||x||=||x,y||+||x,z||$$ for any fixed $y,z \in X$ and $||y,z||\neq 0$. The function  $||.||$ defined on $X$ is a norm on $X$ (\cite{Raymondwfreeseandyeoljecho}). On the other hand, $||x||=max_{y,z} \{||x,y||, ||x,z||\}$ is also a norm on $X$. We see that these norms can be defined for any linearly independent pair of $y, z$ of points in $X$.  It is obvious that any convergent sequence on a $2$-normed space $(X,||.,.||)$ is also convergent on the normed space $(X, ||.||)$, and any quasi-Cauchy sequence in a $2$-normed space $(X,||.,.||)$ is quasi-Cauchy on the normed space $(X, ||.||)$. If a subset $E$ of a $2$-normed space $(X,||.,.||)$ is ward compact, then it is also ward compact in the normed space $(X, ||.||)$. Moreever any ward compact subset of a $2$-normed space $(X, ||.,.||)$ is totally bounded in the normed space $(X, ||.||)$ (\cite[Theorem 3]{CakalliStatisticalquasiCauchysequences}).

Let $(X,||.,.||)$ be a linear $2$-normed space. For $x,z\in X$, let $p_z(x)=||x,z||$. Then, for each $z\in X$, $p_z$ is a real-valued  function on $X$ such that $p_z(x)=||x,z||\geq 0$, $p_z(\alpha x)=|\alpha|||x,z||=|\alpha|p_z(x)$ and $p_z(x+y)=||x+y,z||=||z,x+y||\leq ||z,x||+||z,y||=||x,z||+||y,z||=p_z(x)+p_z(y)$ for all $\alpha \in \textbf{R}$ and all $x,y\in X$. Thus $p_z$ is a semi-norm for each $z\in X$. For $x\in X$, if $||x,z||=0$ for all $z\in X$ then $x=0$ (\cite[Lemma 1.2]{park}). Thus $0\neq x\in X$ implies that there is some $z\in X$ satisfying $p_z(x)=||x,z||\neq 0$. With this additional condition the family $\{p_z:z\in X\}$ is  to be a separating family of semi-norms.

For $\epsilon>0$ and $z\in X$, let $U_{z,\epsilon}=\{x\in X:p_z(x)<\epsilon\}=\{x\in X: ||x,z||<\epsilon\}$. Let $  \mathcal{S}_0=\{U_{z,\epsilon}:\epsilon>0, z\in X\}$ and $\mathcal{B}_0=\{\bigcap \mathcal{F}: \mathcal{F}$ is a finite sub-collection of $ S_0\}$. Define a topology $\tau$ on $X$ by saying that a set $U$ is open if and only if for every $x\in U$ there is some $N\in \mathcal{B}_0$ such that $x+N=\{x+y:y\in N\}\subset U$. That is, $\tau$ is the topology on $X$ that has as a subbase the sets $$\{x\in X: p_z(x-x_0)<\epsilon\},\ z\in X, \ x_0\in X, \ \epsilon>0.$$ The topology $\tau$ gives $X$ the structure of topological vector space. Since the collection $\mathcal{B}_0$ is a local base whose members are convex, $X$ is locally convex (\cite{park}).

Now we introduce a definition of ward continuity in a $2$-normed space.
\begin{defn} A function $f:X\rightarrow X$ is called ward continuous if it preserves quasi-Cauchy sequences, i.e.
$lim_{n\rightarrow\infty}||\Delta f (x_{n}), w||=0$ for every $w \in {X}$
whenever
$lim_{n\rightarrow\infty}||\Delta x_n,z||=0$
for every $z \in X$.
\end{defn}
In connection with quasi-Cauchy sequences and convergent sequences the problem arises to investigate the following types of continuity of functions on
$X$.
\begin{enumerate}
	\item $(x_n)$ is quasi-Cauchy $\Rightarrow(f(x_n))$ is quasi-Cauchy.
	\item $(x_n)$ is quasi-Cauchy $\Rightarrow(f(x_n))$ is convergent.
	\item $(x_n)$ is convergent $\Rightarrow(f(x_n))$ is convergent.
	\item $(x_n)$ is convergent $\Rightarrow(f(x_n))$ is quasi-Cauchy.
\end{enumerate}

It is obvious that $(2)\Rightarrow (1)$ while $(1)$ does not imply $(2)$, $(1)\Rightarrow (4)$ while $(4)$ does not imply $(1)$, $(2)\Rightarrow (3)$ while $(3)$ does not imply $(2)$ and lastly $(3)$ is equivalent to $(4)$. We see that (1) is ward continuity of $f$ and (3) is ordinary continuity of $f$. We give the definition of sequential continuity in $2$-normed spaces in the following before proving that (1) implies (3).

\begin{defn}
A function $f$ on a subset $E$ of a $2$-normed space $(X,||.,.||)$ is said to be sequentially continuous at $x_0$ if for any sequence $(x_{n})$  of points in $E$ converging to $x_{0}$, we have $(f(x_{n}))$ converges to $f(x_{0})$ (see \cite{lael} for the definition for the special case when $f$ is linear).
\end{defn}

\begin{thm}\label{Theoremwardcontinuousimpliessequentiallycontinuous}
If $f:X\rightarrow X$ is ward continuous on a subset $E$ of $X$, then it is sequentially continuous on $E$.
\end{thm}
\begin{proof}
Let $(x_n)$ be any convergent sequence of points in $E$ with
$$lim_{n\rightarrow\infty}||x_n-x_0,y||=0,$$ for all $y\in X$.
Then the sequence $\boldsymbol{\xi}=(\xi_{n})$ defined by
$$
 \xi_{n} =
 \begin{cases}
 x_{k} \;, & \text{if }n=2k-1\text{ for a positive integer}\; k \\
 x_{0} \;, & \text{if }n\text{ is even}
 \end{cases}
 $$
is also convergent to $x_{0}$. Therefore it is a quasi-Cauchy sequence. As $f$ is ward continuous on $E$, the transformed sequence $f(\boldsymbol{\xi})=(f(\xi_{n}))$ obtained by
$$
 f(\xi_{n}) =
 \begin{cases}
 f(x_{k}) \;, & \text{if }n=2k-1\text{ for a positive integer}\; k \\
 f(x_{0}) \;, & \text{if }n\text{ is even}
 \end{cases}
 $$
is also quasi-Cauchy. Now it follows that $\lim_{n\rightarrow\infty}||f(x_n)-f(x_0), z||=0$ for every $z\in{X}$. It implies that the sequence $(f(x_n))$ converges to $f(x_0)$. This completes the proof of the theorem.
\end{proof}
The converse of this theorem is not always true. We give the following example.
\begin{exmp}
Let $f(x,y)=(x^{2},y^{2})$ be the function from $\mathbf{R}\times \mathbf{R}$ into $\mathbf{R}\times \mathbf{R}$. We define a two norm by $||(a_1,a_2),(b_1,b_2)||=\left|a_{1}b_{2}-a_{2}b_{1}\right|$. $\forall z=(z_1,z_2)\in \mathbf{R}\times \mathbf{R}$
\begin{align*}
&lim_{n\rightarrow\infty}||\Delta x_n,z||
=lim_{n\rightarrow\infty}|| x_{n+1}-x_n,z||\\
&=lim_{n\rightarrow\infty}|| (\sqrt{n+1},\sqrt{n+1})-(\sqrt{n},\sqrt{n}),(z_1,z_2)||\\
&=lim_{n\rightarrow\infty}|| (\sqrt{n+1}-\sqrt{n},\sqrt{n+1}-\sqrt{n}),(z_1,z_2)||\\
&=0
\end{align*}
Thus the sequence $x_n=(\sqrt{n},\sqrt{n})$ is quasi-Cauchy. On the other hand
$(f(x_n))=(f(\sqrt{n},\sqrt{n}))=(n,n)$ is not quasi-Cauchy. Since  we have for $w=(1,2)$
\begin{align*}
&lim_{n\rightarrow\infty}||\Delta f( x_n), w||=lim_{n\rightarrow\infty}||f(x_{n+1})-f(x_n), w||\\
&=lim_{n\rightarrow\infty}||(n+1,n+1)-(n,n), w||\\
&=lim_{n\rightarrow\infty}||(1,1),(1,2)||=1\neq 0
\end{align*}
\end{exmp}
Observing that any normed space (single normed) is a first countable topological Hausdorff group, we have the following result related to $G$-sequential continuity (see \cite{CakalliOnGcontinuity}, \cite{CakalliSequentialdefinitionsofcompactness} for the definition of $G$-sequential continuity.).
\begin{cor}
Let $X$ be a finite dimensional $2$-normed space with dimension $m$, and $\{e_{1}, e_{2}, ..., e_{m}\}$ be a base of $X$. Consider the norm defined by $||x||_{\infty}=max_{i=1,2,...,m}\;||x, e_{i}||$.
If $f:X\rightarrow X$ is ward continuous on a subset $E$ of the two normed space $X$, then it is $G$-sequentially continuous on $E$ in the normed space $X$ for every regular subsequential method $G$.
\end{cor}

\begin{cor}
Let $X$ be a finite dimensional $2$-normed space with dimension $m$, and $\{e_{1}, e_{2}, ..., e_{m}\}$ be a base of $X$. Consider the norm defined by $||x||_{\infty}=max_{i=1,2,...,m}\;||x, e_{i}||$.
If $f:X\rightarrow X$ is ward continuous on a subset $E$ of the two normed space $X$, then it is sequentially continuous on $E$ in the normed space $(X, ||.||_{\infty})$.
\end{cor}

The following theorem guarantees that the ward continuous image of a ward compact subset of $X$ is ward compact.

\begin{thm} Let $E$ be a ward compact subset of $X$. If  $f:X\rightarrow X$ is ward continuous on $E$, then $f(E)$ is ward compact.
\end{thm}
\begin{proof}
Ward compactness of $E$ implies that there is a subsequence $z=(z_k)$ of $\bold x=(x_n)$ with $lim_{k\rightarrow\infty}||\Delta z_k,y||=0, \forall y\in X.$ Let $(t_k)=(f(z_k))$. $(t_k)$ is the subsequence of the sequence $f(x)$ with $lim_{k\rightarrow\infty}||\Delta t_k,f(y)||=0.$ This completes the proof of the theorem.
\end{proof}
\begin{defn}
A function $f:X\rightarrow X$ is called uniformly continuous on a subset $E$ of $X$ if for any $\varepsilon>0$ and for any $w\in {X}$ there exist positive real numbers $\delta_1, \delta_2,..., \delta_p$ and $z_1, z_2, ..., z_p\in {X}$ such that $||f(x)-f(y),w||<\epsilon$ whenever $||x-y,z_k||<\delta_{k}$, \; $k=1, 2, ..., p$.
\end{defn}

We note that a function $f:X\rightarrow X$ is called uniformly continuous on a subset $E$ of $X$ if for any finite choose of $\epsilon_1, \epsilon_2,..., \epsilon_m$ and $w_1,w_2,...,w_m$ there exist $\delta_1, \delta_2,..., \delta_p$ and $z_1, z_2, ..., z_p$ such that $||f(x)-f(y),w_j||<\epsilon_j$ for $j=1, 2, ..., m$ whenever $||x-y,z_k||<\delta_{k}$, \; $k=1, 2, ..., p$. On the other hand, taking one $\epsilon$ and one $\delta$ in the consideration

we give the following definition.
\begin{defn}
A function $f:X\rightarrow X$ is called $u$-continuous on a subset $E$ of $X$ if for any given $\epsilon>0$ there exists a $\delta>0$ such that $||f(x)-f(y), w||<\epsilon$ for any $w\in {X}$ whenever $||x-y,z||<\delta$ for any $x,y\in{E}$, and $z\in X$.
\end{defn}
\begin{thm} \label{Theoremucontinuousfunctioniswardcontinuous}
If a function $f:X\rightarrow X$ is $u$-continuous on a subset $E$ of $X$, then it is ward continuous on $E$.
\end{thm}
\begin{proof}
Let $f$ be $u$-continuous on $E$, and $(x_n)$ be any quasi-Cauchy sequence of points in $E$. Then for given any $\epsilon>0$ there exists a $\delta>0$ such that $||f(x)-f(y), w||<\epsilon$ for any $w\in {X}$  whenever $||x-y,z||<\delta$ for any $x,y,z\in X$. For the choice of $\delta>0$, there exists an $N=N(\delta)=N_1(\epsilon)\in N$ such that $||\Delta x_n,z||<\delta$ for every $n>N$.
So  $||\Delta f(x_n),f(z)||<\epsilon$
for every $n>N$. This completes the proof of the theorem.
\end{proof}
\begin{cor} Any $u$-continuous function is sequentially continuous.
\end{cor}
\begin{proof} Proof follows easily from Theorem \ref{Theoremwardcontinuousimpliessequentiallycontinuous}, so is omitted.
\end{proof}

\begin{cor} $u$-continuous image of any ward compact subset of $X$ is ward compact.
\end{cor}
Any $u$-continuous function is uniformly-continuous. But the converse is not always true. It is not difficult to construct an example in two dimensional case.
\begin{thm} \label{Theoremuniformlycontinuousfunctioniswardcontinuous}
If a function $f:X\rightarrow X$ is uniformly continuous on a subset $E$ of $X$, then it is ward continuous on $E$.
\end{thm}
\begin{proof}
Although the proof can be obtained by Theorem \ref{Theoremucontinuousfunctioniswardcontinuous} we give a proof for completeness. Let $f$ be uniformly continuous on $E$, and $(x_n)$ be any quasi-Cauchy sequence of points in $E$. Choose any $\epsilon>0$ and any $w \in {X}$.  As $f$ is uniformly continuous on $E$ for this $\epsilon$, and $w$
there exist $\delta_1, \delta_2,..., \delta_p$ and $z_1, z_2, ..., z_p$ such that $||f(x)-f(y),w||<\epsilon$ whenever $||x-y,z_k||<\delta_{k}$, $k=1, 2, ..., p$. Since $(x_n)$ is a quasi-Cauchy sequence, for $\delta_{1}$ there exists a positive integer $n_{1}$ such that
$||x_{n+1}-x_{n},z_{1}||<\delta_{1}$. Similarly for $\delta_{2}$ there exists a positive integer $n_{2}$ such that
$||x_{n+1}-x_{n},z_{2}||<\delta_{2}$. Having found positive integers $n_{3}$, $n_{4}$, .., and so on we finally can find a positive integer $n_{p}$ such that
$||x_{n+1}-x_{n},z_{p}||<\delta_{p}$. Now write $n_{0}=max\{n_{1}, n_{2}, n_{3}, ..., n_{p}\}$. Then  $||f(x_{n+1})-f(x_{n}),w||<\epsilon$ for $n \geq n_{0}$. Hence $(f(x_{n}))$ is a quasi-Cauchy sequence, therefore $f$ preserves quasi-Cauchy sequences. This completes the proof of the theorem.
\end{proof}

It is a well known result that uniform limit of a sequence of continuous functions is continuous.  Some related results in $2$-normed spaces are obtained in \cite{SarabadanandTalebiStatisticalconvergenceandidealconvergenceofsequencesoffunctionsin$2$-normedspaces}. For the case $u$-continuity, we have the following.

\begin{thm}
If $(f_n)$ is a sequence of $u$-continuous functions defined on a subset $E$ of $X$, and $(f_n)$ is uniformly convergent to a function $f$, then $f$ is $u$-continuous on $E$.
\end{thm}
\begin{proof}
Since $(f_n)$ is uniformly convergent to $f$, for given any $\epsilon>0$ there exists a positive integer $N$ such that $||f_n(x)-f(x), w||<\frac{\epsilon}{3}$ whenever $n\geq N$ for all $x, w\in E$. Since $f_N$ is $u$-continuous on $E$, for $\frac{\varepsilon}{3}$, there exists a positive real number $\delta$ such that $||f_{N}(x)-f_{N}(y), w||<\frac{\epsilon}{3}$ for any $w\in {X}$ whenever $||x-y, z||<\delta$ for any $x, y, z \in {X}$. Then whenever $||x-y, z||<\delta$, we have
\begin{align*}
&||f(x)-f(y), w|| \leq||f(x)-f_N(x), w||+||f_N(x)-f_N(y), w||+||f_N(y)-f(y), w||\\
&\leq \frac{\epsilon}{3}+\frac{\epsilon}{3}+\frac{\epsilon}{3}=\epsilon
\end{align*}
for every $w \in {X}$. So $f$ is $u$-continuous on $E$, and the proof is completed.
\end{proof}

In the real case it was proved that uniform limit of a sequence of ward continuous functions is ward continuous (\cite{cakalli}). It is also true in $2$-normed spaces.

\begin{thm}
 If $(f_n)$ is a sequence of ward continuous functions defined on a subset $E$ of $X$, and $(f_n)$ is uniformly convergent to a function $f$, then $f$ is ward continuous on $E$.
 
\end{thm}
\begin{proof} To prove that $f$ is ward continuous on $E$, take any quasi-Cauchy sequence $(x_{n})$ of points in $E$. Let $\varepsilon$ be any positive real number. Since $(f_n)$ is uniformly convergent to $f$, for  $\frac{\epsilon}{3}>0$ there exists a positive integer N such that $||f_n(x)-f(x),z||<\frac{\epsilon}{3}$ whenever $n\geq N$ for all $x,z\in E$. Since $f_N$ is ward continuous on $E$, there exists a positive integer $N_1\geq N$ such that $||f_N(x_{n+1})-f_N(x_{n}),z||<\frac{\epsilon}{3}$ for $n\geq N_1$. Then we have for $n\geq N_1$
\begin{align*}
&||f(x_{n+1})-f(x_{n}),z||\\
&\leq||f(x_{n+1})-f_N(x_{n+1}),z||+||f_N(x_{n+1})-f_N(x_{n}),z||+||f_N(x_{n})-f(x_{n}),z||\\
&\leq \frac{\epsilon}{3}+\frac{\epsilon}{3}+\frac{\epsilon}{3}=\epsilon
\end{align*}
for every $z \in {X}$. So $f$ is ward continuous on $E$, and the proof is completed.
\end{proof}

\section{Conclusion}

The concept of $2$-normed spaces was extensively studied by G\"ahler who asked what the real motivation for studying $2$-norm structure is, and if there is a physical situation or an abstract concept where norm topology does not work but $2$-norm topology does work. We observe that if a term in the definition of  $2$-norm represents the change of a shape, and the  $2$-norm stands for the associated area, we can think of some plausible application of the notion of  $2$-norm, and then the generalized convergence make sense. This can also be viewed as: suppose for a particular output we need  two-inputs but with one  main input and other input is required to complete the process. So that one may expect to be a more useful tool in the field of $2$-normed space in modelling various problems, occuring in many areas of science. computer science and information theory. Such possible applications attract researchers to be involved in investigation on $2$-normed spaces. We note that the present work contains not only an investigation of quasi-Cauchy sequences as it has been presented in a very different setting, i.e. in a two normed space which is quite different from the real case, and metric case, but also an investigation of some other kinds of sequential continuities.

We note that the study in this paper can be carried to $n$-normed spaces without any difficulty (see for example \cite{ReddyandDuttaOnEquivalenceofnNorms} for the definition of an $n$-normed space). For further study, we suggest to investigate quasi-Cauchy sequences of points and fuzzy functions in a $2$-normed fuzzy spaces. However due to the change in
settings, the definitions and methods of proofs will not always be analogous to those of the present work (for example see \cite{CakalliandPratulFuzzycompactnessviasummability}, \cite{SomasundaramandBeaulajSomeaspectsof$2$-fuzzy$2$-normedlinearspaces}, and \cite{KocinacSelectionpropertiesinfuzzymetricspaces}).
For another further study we suggest to investigate quasi-Cauchy sequences of double sequences in two normed spaces (see for example \cite{DasandSavasandBhuniaSantanuTwovaluedmeasureandsomenewdoublesequencespacesin2-normedspaces}, and \cite{MohiuddineSomenewresultsonapproximationinfuzzy2-normedspaces} for the definitions and related concepts in the double case).

\end{document}